\documentclass[reqno]{amsart}
\usepackage{amsfonts}
\usepackage{amsmath}
\usepackage{graphicx}
\usepackage{amssymb}

\newtheorem{theorem}{Theorem}
\theoremstyle{plain}

\newtheorem{definition}{Definition}

\newtheorem{lemma}{Lemma}

\newtheorem{remark}{Remark}

\numberwithin{equation}{section}
\input{tcilatex}

\begin{document}
\title[Superquadratic functions and averages]{Superquadratic functions and
refinements of inequalities between averages}
\author{S. Abramovich}
\address{Department of Mathematics\\
University of Haifa\\
Haifa, 31905, Israel}
\email{abramos@math.haifa.ac.il}
\author{J. Bari\'{c}}
\address{FESB, University of Split\\
Ru\dj era Bo\v{s}kovi\'{c}a b.b., 21000 Split\\
Croatia}
\email{jbaric@fesb.hr}
\author{M. Mati\'{c}}
\address{Department of Mathematics\\
Faculty of Natural Sciences, Mathematics and Education\\
University of Split\\
Teslina 12, 21000 Split\\
Croatia}
\email{mmatic@pmfst.hr}
\author{J. Pe\v{c}ari\'{c}}
\address{Faculty of Textile Technology\\
University of Zagreb\\
Pierottijeva 6, 10000 Zagreb}
\email{pecaric@element.hr}
\thanks{}
\date{October 24, 2011}
\dedicatory{ }

\begin{abstract}

In this paper upper bounds are given for the successive differences $A_{n+1}-A_{n}$ and B$_{n}-B_{n-1}$
where $A_{n}=1/\left( n-1\right) \tsum_{r=1}^{n-1}f\left( r/n\right)$, 
$B_{n}=1/\left( n+1\right) \tsum_{r=0}^{n}f\left( r/n\right)$ and $f$ is superquadratic function.
We obtain bounds for the successive differences of the more general sequence
$1/c_{n}\tsum_{r=1}^{n}f\left( a_{r}/b_{n}\right)$ when $f$ is superquadratic, 
which refine known results for convex functions.
We also obtain bounds for various successive differences when $f$ is an
increasing subquadratic function.
\end{abstract}

\maketitle

\section{Introduction}

We define the averages%
\begin{equation*}
A_{n}(f)=\tfrac{1}{n-1}\tsum_{r=1}^{n-1}f\left( \tfrac{r}{n}\right) ,\quad
n\geq 2
\end{equation*}%
and%
\begin{equation*}
B_{n}\left( f\right) =\tfrac{1}{n+1}\tsum_{r=0}^{n}f\left( \tfrac{r}{n}\right)
,\quad n\geq 1.
\end{equation*}%
In \cite{BJ} (see also \cite{B}) it was shown that if $f$ is convex, then $%
A_{n}(f)$\ increases with $n$ and $B_{n}\left( f\right) $\ decreases. In 
\cite{AJS2}, for the class of superquadratic functions, 
the theorems of \cite{BJ} are generalized in the following Theorem
A and Theorem B.

\medskip \textbf{Theorem A} \cite{AJS2} \textit{%
If }$f$\textit{\ is superquadratic on }$\left[ 0,1\right] $\textit{\ then
for }$n\geq 2$%
\begin{equation*}
A_{n+1}\left( f\right) -A_{n}\left( f\right) \geq f\left( \tfrac{1}{3n}%
\right) +\tsum_{r=1}^{n-1}\lambda _{r}f\left( y_{r}\right) ,
\end{equation*}%
\textit{where }$y_{r}=\tfrac{\left\vert 2n-1-3r\right\vert }{3n\left(
n+1\right) }.$ \textit{Moreover if }$f$\textit{\ \ is superquadratic and
non-negative, then for }$n\geq 3$%
\begin{equation*}
A_{n+1}\left( f\right) -A_{n}\left( f\right) \geq f\left( \tfrac{1}{3n}%
\right) +f\left( \tfrac{16}{81\left( n+3\right) }\right) .
\end{equation*}

\textbf{Theorem B} \cite{AJS2} \textit{If }$f$%
\textit{\ is superquadratic on} $[0,1]$, \textit{then for }$n\geq 2$%
\begin{equation*}
B_{n-1}\left( f\right) -B_{n}\left( f\right) \geq f\left( \tfrac{1}{3n}%
\right) +\tsum_{r=1}^{n}\lambda _{r}f\left( y_{r}\right) ,
\end{equation*}%
\textit{where }$y_{r}=\tfrac{\left\vert 2n+1-3r\right\vert }{3n\left(
n-1\right) }.$\textit{\ Moreover, if }$f$\textit{\ is also non-negative,
then for }$n\geq 2$%
\begin{equation*}
B_{n-1}\left( f\right) -B_{n}\left( f\right) \geq f\left( \tfrac{1}{3n}%
\right) +f\left( \tfrac{16}{81n}\right) .
\end{equation*}

\medskip In this article we find upper bounds for the difference $%
A_{n+1}\left( f\right) -A_{n}\left( f\right) $\ and for the difference $%
B_{n-1}\left( f\right) -B_{n}\left( f\right) .$\ \ We also generalize 
the lower bounds obtained in \cite{CQCD}, \cite{EP} 
and \cite{QG}.

Now we present the class of superquadratic functions $f$ that we use in
this paper to get our results related to differences of averages. This was
introduced in \cite{AJS1} and\ \cite{AJS2}, and dealt with in \cite{ABKB}, 
\cite{ABM}, \cite{ABMP}, \cite{BPV}, \cite{BV}, \cite{GT}, \cite{OP} and
other papers.

\begin{definition}
\label{D1}\cite{AJS1,AJS2} A function $f$, defined on an
interval \ $I=\left[ 0,L\right] $ or $\left[ 0,\infty \right) $ is
superquadratic, if for each $x$ in $I$, there exists a real number $C\left(
x\right) $ such that 
\begin{equation}
f\left( y\right) -f\left( x\right) \geq C\left( x\right) \left( y-x\right)
+f\left( \left\vert y-x\right\vert \right)  \label{1.1}
\end{equation}%
for all $y\in I.$ A function $f$ \ is subquadratic if $-f$ is superquadratic.
\end{definition}

As stated in the following Lemma A, positive superquadratic functions are
also convex, increasing and satisfy $f\left( 0\right) =0$ (like $f\left(
x\right) =x^{m},$ $m\geq 2$). \ Therefore the results obtained in this paper
leads to refinements of results in \cite{BJ}, \cite{CQCD} and \cite{QG}.

% If on the other hand $f$ is a positive subquadratic function, we get bounds
% for $S_{n}\left( f\right) -S_{n+1}\left( f\right) $ for functions which are
% not convex and not concave like 
% \begin{equation*}
% f\left( x\right) =\left\{ 
% \begin{array}{cc}
% x^{2}-2x^{2}\log x, & x>0 \\ 
% 0, & x=0%
% \end{array}%
% \right. .
% \end{equation*}
\medskip

\textbf{Lemma A} \cite{AJS1} \textit{Let }$f$\textit{\ be a
superquadratic function with }$C\left( x\right) $\textit{\ as in Definition %
\ref{D1}.} Then \newline
\textit{(i) \ }$f\left( 0\right) \leq 0,$\newline
\textit{(ii) \ if }$f\left( 0\right) =f^{\prime }\left( 0\right) =0,$\textit{%
\ then }$C\left( x\right) =f^{\prime }\left( x\right) $\textit{\ whenever }$%
f $\textit{\ is differentiable at }$x>0$,\newline
\textit{(iii) if }$f\left( x\right) \geq 0$\textit{, }$x\in I$,\textit{\ then 
}$f$\textit{\ is convex on }$I$ \textit{and }$f\left( 0\right) =f^{\prime
}\left( 0\right) =0$\textit{\ .}

\medskip

\textbf{Lemma B} \cite{AJS2} \textit{Suppose
that }$f$\textit{\ is superquadratic. Let }$x_{r}\geq 0,$\textit{\ }$1\leq
r\leq n$\textit{\ \ and let }$\overline{x}=\tsum_{r=1}^{n}\lambda _{r}x_{r}$%
\textit{\ where }$\lambda _{r}\geq 0,$\textit{\ and }$\tsum_{r=1}^{n}\lambda
_{r}=1.$\textit{\ Then }%
\begin{equation}
\tsum_{r=1}^{n}\lambda _{r}f\left( x_{r}\right) \geq f\left( \overline{x}%
\right) +\tsum_{r=1}^{n}\lambda _{r}f\left( \left\vert x_{r}-\overline{x}%
\right\vert \right) .  \label{1.2}
\end{equation}%
\textit{If \ }$f\left( x\right) $\textit{\ is subquadratic the reverse
inequality to} (\ref{1.2}) \textit{holds}.

\medskip

From Lemma B we get the immediate result which we state in the following
Lemma 1, by repeating (\ref{1.2}) $t$ times.

\medskip

\begin{lemma}
\label{L1} Let $f$\ be superquadratic on \ $\left[ 0,L\right] $\ and let \ $%
x,y\in \left[ 0,L\right] $, $0\leq \lambda \leq 1$. Then 
\begin{eqnarray}
&&\lambda f\left( x\right) +\left( 1-\lambda \right) f\left( y\right)  \notag
\\
&\geq &f\left( \lambda x+\left( 1-\lambda \right) y\right) +\lambda f\left(
\left( 1-\lambda \right) \left\vert y-x\right\vert \right) +\left( 1-\lambda
\right) f\left( \lambda \left\vert y-x\right\vert \right)  \notag \\
&\geq &f\left( \lambda x+\left( 1-\lambda \right) y\right)
+\tsum_{k=0}^{t-1}f\left( 2\lambda \left( 1-\lambda \right) \left\vert
1-2\lambda \right\vert ^{k}\left\vert x-y\right\vert \right) +  \notag \\
&&\lambda f\left( \left( 1-\lambda \right) \left\vert 1-2\lambda \right\vert
^{t}\left\vert x-y\right\vert \right) +\left( 1-\lambda \right) f\left(
\lambda \left\vert 1-2\lambda \right\vert ^{t}\left\vert x-y\right\vert
\right) .  \label{1.3}
\end{eqnarray}%
If $f$ is positive superquadratic we get from (\ref{1.3}) that%
\begin{eqnarray}
&&\lambda f\left( x\right) +\left( 1-\lambda \right) f\left( y\right)  \notag
\\
&\geq &f\left( \lambda x+\left( 1-\lambda \right) y\right)
+\tsum_{k=0}^{t-1}f\left( 2\lambda \left( 1-\lambda \right) \left\vert
1-2\lambda \right\vert ^{k}\left\vert x-y\right\vert \right)  \label{1.4}
\end{eqnarray}%
If $f$ is subquadratic we get the reverse inequality of (\ref{1.3}). 

% \begin{eqnarray}
% &&\lambda f\left( x\right) +\left( 1-\lambda \right) f\left( y\right)  \notag
% \\
% &\leq &f\left( \lambda x+\left( 1-\lambda \right) y\right)
% +\tsum_{k=0}^{t-1}\left( f\left( 2\lambda \left( 1-\lambda \right)
% \left\vert 1-2\lambda \right\vert ^{k}\left\vert x-y\right\vert \right)
% \right)  \notag \\
% &&+\lambda f\left( 1-\lambda \right) \left\vert 1-2\lambda \right\vert
% ^{t}\left\vert x-y\right\vert +\left( 1-\lambda \right) f\left( \lambda
% \left\vert 1-2\lambda \right\vert ^{t}\left\vert x-y\right\vert \right) .
% \label{1.5}
% \end{eqnarray}
\end{lemma}

From Lemma 1 we get the following result for a positive superquadratic
functions.

\begin{lemma}
\label{L2}Let $f$ be a positive superquadratic function on $\left[ 0,L\right]
$. Let $A_{i}\in \left[ 0,L\right] ,$ $0\leq \lambda _{i}\leq 1,$ $%
i=1,...,m. $ Then%
\begin{eqnarray}
&&\tsum_{i=1}^{m}\left[ \lambda _{i}f\left( \left( 1-\lambda _{i}\right)
A_{i}\right) +\left( 1-\lambda _{i}\right) f\left( \lambda _{i}A_{i}\right) %
\right]  \notag \\
&\geq &\tsum_{k=0}^{t}mf\left( \tfrac{\tsum_{i=1}^{m}2\lambda _{i}\left(
1-\lambda _{i}\right) \left\vert 1-2\lambda _{i}\right\vert ^{k}A_{i}}{m}%
\right) ,\qquad t=0,1,2,...  \label{1.6}
\end{eqnarray}%
If $\lambda _{i}A_{i}\geq A,$ \ $i=1,...,m,$ \ then%
\begin{eqnarray}
&&\tsum_{i=1}^{m}\left[ \lambda _{i}f\left( \left( 1-\lambda _{i}\right)
A_{i}\right) +\left( 1-\lambda _{i}\right) f\left( \lambda _{i}A_{i}\right) %
\right]  \notag \\
&\geq &\tsum_{k=0}^{t}mf\left( \tfrac{\tsum_{i=1}^{m}2\left( 1-\lambda
_{i}\right) \left\vert 1-2\lambda _{i}\right\vert ^{k}A}{m}\right) .
\label{1.7}
\end{eqnarray}
\end{lemma}

% Lemmas 1 and 2 are useful in obtaining bounds for the difference $%
% S_{n}\left( f\right) -S_{n+1}\left( f\right) $ for superquadratic functions
% and for subquadratic functions.

\section{\protect\medskip Superquadracity, subquadracity and upper bounds of
averages}

% In \cite{AJS2} Theorem A is proved for 
% \begin{equation*}
% A_{n}\left( f\right) :=\tfrac{1}{n-1}\tsum_{r=1}^{n-1}f\left( \tfrac{r}{n}%
% \right) ,\text{ \ }n\geq 2,
% \end{equation*}%
% when $f$ is superquadratic function. As said before this theorem generalize
% theorems of \cite{BJ} that dealt with convex function $f$.

\medskip

In the following theorem we establish an upper bounds for the
differences $A_{n+1}\left( f\right) -A_{n}\left( f\right) $  and 
$B_{n-1}\left( f\right) -B_{n}\left( f\right)$ for superquadratic
functions whereas in [6] and [2], lower bounds were established for convex
functions and superquadratic functions respectively.

\begin{theorem}
\label{T4}Let f be a superquadratic function on $\left[ 0,1\right]$. Then
for $1\leq r\leq n,$ $n\geq 3$, we get%
\begin{eqnarray}
&&A_{n+1}\left( f\right) -A_{n}(f)  \notag \\
% &=&\tfrac{1}{n}\tsum_{r=1}^{n}f\left( \tfrac{r}{n+1}\right) -\tfrac{1}{n-1}%
% \tsum_{r=1}^{n-1}f\left( \tfrac{r}{n}\right)  \notag \\
&\leq &\tfrac{1}{2}\left[ f\left( \tfrac{1}{n+1}\right) +f\left( \tfrac{n}{%
n+1}\right) \right] -\tsum_{r=1}^{n-1}\left[ \tfrac{2r}{n\left( n-1\right) }%
f\left( \tfrac{n-r-1}{n+1}\right) +\tfrac{1}{n-1}f\left( \tfrac{r}{n}\right) %
\right] .  \label{2.29}
\end{eqnarray}%
Moreover, if $f$ is also positive, then%
\begin{equation}
A_{n+1}\left( f\right) -A_{n}\left( f\right) \leq \tfrac{1}{2}\left[ f\left( 
\tfrac{1}{n+1}\right) +f\left( \tfrac{n}{n+1}\right) \right] -\left[ f\left( 
\tfrac{n-2}{3\left( n+1\right) }\right) +f\left( \tfrac{1}{2}\right) \right]
\label{2.30}
\end{equation}
\end{theorem}

\begin{proof}
As $f$ is a superquadratic function on $\left[ 0,1\right] ,$ then by
inserting in Lemma \ref{L1} $\ x=\tfrac{1}{n+1},\ y=\tfrac{n}{n+1},$ $\lambda =%
\tfrac{n-r}{n-1},$ $r=1,...,n$ \ we get%
\begin{equation}
f\left( \tfrac{r}{n+1}\right) \leq \tfrac{n-r}{n-1}f\left( \tfrac{1}{n+1}%
\right) +\tfrac{r-1}{n-1}f\left( \tfrac{n}{n+1}\right) -\left[ \tfrac{n-r}{%
n-1}f\left( \tfrac{r-1}{n+1}\right) +\tfrac{r-1}{n-1}f\left( \tfrac{n-r}{n+1}%
\right) \right] .  \label{2.31}
\end{equation}%
Hence 
\begin{eqnarray}
&&\tsum_{r=1}^{n}f\left( \tfrac{r}{n+1}\right)  \notag \\
&\leq &\tsum_{r=1}^{n}\left[ \tfrac{n-r}{n-1}f\left( \tfrac{1}{n+1}\right) +%
\tfrac{r-1}{n-1}f\left( \tfrac{n}{n+1}\right) \right] -\tsum_{r=1}^{n}\left[ 
\tfrac{n-r}{n-1}f\left( \tfrac{r-1}{n+1}\right) +\tfrac{r-1}{n-1}f\left( 
\tfrac{n-r}{n+1}\right) \right]  \notag \\
&=&\tfrac{n}{2}\left[ f\left( \tfrac{1}{n+1}\right) +f\left( \tfrac{n}{n+1}%
\right) \right] -\tsum_{r=1}^{n}\left[ \tfrac{n-r}{n-1}f\left( \tfrac{r-1}{%
n+1}\right) +\tfrac{r-1}{n-1}f\left( \tfrac{n-r}{n+1}\right) \right]  \notag
\\
&=&\tfrac{n}{2}\left[ f\left( \tfrac{1}{n+1}\right) +f\left( \tfrac{n}{n+1}%
\right) \right] -\tsum_{r=1}^{n-1}\tfrac{2r}{n-1}f\left( \tfrac{n-r-1}{n+1}%
\right) .  \label{2.32}
\end{eqnarray}%
From (\ref{2.32}) we get (\ref{2.29}).

If the superquadratic $f$ is also positive on $\left[ 0,1\right] $ then $f$
is convex and from (\ref{2.29}) we get
\begin{eqnarray*}
A_{n+1}\left( f\right) -A_{n}\left( f\right) &\leq &\tfrac{1}{2}\left[
f\left( \tfrac{1}{n+1}\right) +f\left( \tfrac{n}{n+1}\right) \right] -\left[
f\left( \tsum_{r=1}^{n-1}\tfrac{2\left( r-1\right) \left( n-r\right) }{%
n\left( n-1\right) \left( n+1\right) }\right) +f\left( \tfrac{1}{2}\right) %
\right] \\
&=&\tfrac{1}{2}\left[ f\left( \tfrac{1}{n+1}\right) +f\left( \tfrac{n}{n+1}%
\right) \right] -\left[ f\left( \tfrac{n-2}{3\left( n+1\right) }\right)
+f\left( \tfrac{1}{2}\right) \right] .
\end{eqnarray*}%
Hence, (\ref{2.30}) holds. This completes the proof of the theorem.
\end{proof}

% Similarily to Theorem A, Theorem B was proved in \cite{AJS2} [5,Th3.2,Th4,2]
% for 
% \begin{equation*}
% B_{n}\left( f\right) :=\tfrac{1}{n+1}\tsum_{r=0}^{n}f\left( \tfrac{r}{n}%
% \right) ,\text{ \ \ }n\geq 1,
% \end{equation*}%
% and it generalizes theorems of \cite{BJ} which dealt with convex functions $f
% $.\medskip 

In the following theorem we establish an upper bound for the difference $%
B_{n-1}\left( f\right) -B_{n}\left( f\right) $ \ by similar reasoning to
those used in Theorem \ref{T4}. The proof is omitted here.

\begin{theorem}
\label{T5}Let $f$ be a positive superquadratic function on $[0,1]$. Then 
\begin{equation*}
B_{n-1}\left( f\right) -B_{n}\left( f\right) \leq \tfrac{n-1}{2n}\left[
f\left( \tfrac{1}{n-1}\right) +f\left( 1\right) \right] -\tfrac{n-3}{n}%
f\left( \tfrac{1}{3}\right) -f\left( \tfrac{1}{2}\right) .
\end{equation*}
\end{theorem}

\begin{remark}
\label{R3} The arguments in Theorems \ref{T4} and \ref{T5} can be generalized
to an upper bound of \ $\tsum_{i=1}^{m}f\left( \tfrac{a_{i}}{a_n}\right) $ where$\
\left( a_{i}\right) _{i\geq 1}\ $is$\ $positive \ increasing sequence and \ $f\ $\ is
superquadratic function. Therefore, putting in Lemma \ref{L1} that  \ $\lambda =%
\tfrac{a_{n}-a_{i}}{a_{n}-a_{1}},$ $x=\tfrac{a_{i}}{a_{n}},$ $y=\tfrac{a_{n}%
}{a_{n}}=1,$ it follows
\begin{equation*}
f\left( \tfrac{a_{i}}{a_{n}}\right) \leq \tfrac{a_{n}-a_{i}}{a_{n}-a_{1}}%
f\left( \tfrac{a_{1}}{a_{n}}\right) +\tfrac{a_{i}-a_{1}}{a_{n}-a_{1}}f\left(
1\right) -\tfrac{a_{n}-a_{i}}{a_{n}-a_{1}}f\left( \tfrac{a_{i}-a_{1}}{a_{n}}%
\right) -\tfrac{a_{i}-a_{1}}{a_{n}-a_{1}}f\left( \tfrac{a_{n}-a_{i}}{a_{n}}%
\right) .
\end{equation*}%
Hence,

\begin{eqnarray*}
\tsum_{i=1}^{m}f\left( \tfrac{a_{i}}{a_{n}}\right) &\leq &f\left( \tfrac{%
a_{1}}{a_{n}}\right) \tsum_{i=1}^{m}\tfrac{a_{n}-a_{i}}{a_{n}-a_{1}}+f\left(
1\right) \tsum_{i=1}^{m}\tfrac{a_{i}-a_{1}}{a_{n}-a_{1}} \\
&&-\tsum_{i=1}^{m}\left[ \tfrac{a_{n}-a_{i}}{a_{n}-a_{1}}f\left( \tfrac{%
a_{i}-a_{1}}{a_{n}}\right) +\tfrac{a_{i}-a_{1}}{a_{n}-a_{1}}f\left( \tfrac{%
a_{n}-a_{i}}{a_{n}}\right) \right] .
\end{eqnarray*}%
If $f$ is also positive and therefore convex, we get from the last
inequality that 
\begin{eqnarray}
\tsum_{i=1}^{m}f\left( \tfrac{a_{i}}{a_{n}}\right) &\leq &f\left( \tfrac{%
a_{1}}{a_{n}}\right) \tsum_{i=1}^{m}\tfrac{a_{n}-a_{i}}{a_{n}-a_{1}}+f\left(
1\right) \tsum_{i=1}^{m}\tfrac{a_{i}-a_{1}}{a_{n}-a_{1}}  \notag \\
&&-mf\left( \tsum_{i=1}^{m}\tfrac{2\left( a_{n}-a_{i}\right) \left(
a_{i}-a_{1}\right) }{\left( a_{n}-a_{1}\right) a_{n}m}\right) .  \label{2.33}
\end{eqnarray}%

In theorems \ref{T4} and \ref{T5} we simplified the last inequality
according to the specific $m$, $n,$ and $\left( a_{i}\right) _{i\geq 1}.$

If $a_{i},$ $i=1,...,n$ is a general positive increasing sequence we get
from (\ref{2.33}), using $2\left( a_{n}-a_{i}\right) \left(
a_{i}-a_{1}\right) >2\left( a_{n}-a_{n-1}\right) \left( a_{2}-a_{1}\right) ,$
that for $m=n$ \ the inequality 
\begin{eqnarray}
\tsum_{i=1}^{n}f\left( \tfrac{a_{i}}{a_{n}}\right) &\leq &f\left( \tfrac{%
a_{1}}{a_{n}}\right) \left( \tfrac{na_{n}-\tsum_{i=1}^{n}a_{i}}{a_{n}-a_{1}}%
\right) +f\left( 1\right) \left( \tfrac{\tsum_{i=1}^{n}a_{i}-na_{1}}{\left(
a_{n}-a_{1}\right) }\right)  \notag \\
&&-nf\left( \tfrac{2\left( a_{n}-a_{n-1}\right) \left( a_{2}-a_{1}\right) }{%
\left( a_{n}-a_{1}\right) a_{n}}\right)  \label{2.34}
\end{eqnarray}%
holds. Since for convex functions we have%
\begin{equation}
\tsum_{i=1}^{n}f\left( \tfrac{a_{i}}{a_{n}}\right) \geq
n\tsum_{i=1}^{n}f\left( \tfrac{\tsum_{i=1}^{n}a_{i}}{na_{n}}\right) ,
\label{2.35}
\end{equation}
from (\ref{2.34}) and (\ref{2.35}) the following results directly follows.
\end{remark}

\begin{theorem}
\label{T6}Let $f$ be a positive superquadratic function on $\left[ 0,1\right]
$ and let  $(a_{i})_{i\in \mathbb{N}}$ be a positive increasing sequence. Then%
\begin{eqnarray*}
&&\tfrac{1}{n}\tsum_{i=1}^{n}f\left( \tfrac{a_{i}}{a_{n}}\right) -\tfrac{1}{%
n+1}\tsum_{i=1}^{n+1}f\left( \tfrac{a_{i}}{a_{n+1}}\right) \\
&\leq &\tfrac{na_{n}-\tsum_{i=1}^{n}a_{i}}{n\left( a_{n}-a_{1}\right) }%
f\left( \tfrac{a_{1}}{a_{n}}\right) +\tfrac{\left(
\tsum_{i=1}^{n}a_{i}\right) -na_{1}}{n\left( a_{n}-a_{1}\right) }f\left(
1\right) \\
&&-f\left( \tfrac{2\left( a_{n}-a_{n-1}\right) \left( a_{2}-a_{1}\right) }{%
\left( a_{n}-a_{1}\right) a_{n}}\right) -f\left( \tfrac{\tsum_{i=1}^{n+1}a_{i}%
}{\left( n+1\right) a_{n+1}}\right) .
\end{eqnarray*}
\end{theorem}

\begin{theorem}
\label{T7}Let $f$ be a positive superquadratic function on $\left[ 0,1\right]
.$ Let $(a_{i})_{i\in \mathbb{N}}$ be a positive increasing sequence and $(c_{i})_{i\in \mathbb{N}}$ 
be a positive sequence. Then%
\begin{eqnarray*}
&&\tfrac{1}{c_{n}}\tsum_{i=1}^{n}f\left( \tfrac{a_{i}}{a_{n}}\right) -\tfrac{%
1}{c_{n+1}}\tsum_{i=1}^{n+1}f\left( \tfrac{a_{i}}{a_{n+1}}\right) \\
&\leq &\tfrac{na_{n}-\tsum_{i=1}^{n}a_{i}}{\left( a_{n}-a_{1}\right) }\cdot 
\tfrac{f\left( \tfrac{a_{1}}{a_{n}}\right) }{c_{n}}+\tfrac{%
\tsum_{i=1}^{n}a_{i}-na_{1}}{n\left( a_{n}-a_{1}\right) }\cdot \tfrac{f\left(
1\right) }{c_{n}} \\
&&-\tfrac{n}{c_{n}}f\left( \tfrac{2\left( a_{n}-a_{n-1}\right) \left(
a_{2}-a_{1}\right) }{\left( a_{n}-a_{1}\right) a_{n}}\right) -\tfrac{n+1}{%
c_{n+1}}f\left( \tfrac{\tsum_{i=1}^{n+1}a_{i}}{\left( n+1\right) a_{n+1}}%
\right) .
\end{eqnarray*}
\end{theorem}

In the following theorem we prove some inequalities for $A_{n}\left(
f\right) $ and \ $B_{n}\left( f\right) $\ for increasing subquadratic
functions. 

First we state a lemma that follows immediately from Lemma B for
subquadratic functions.

\begin{lemma}
\label{L3}Let $f$ be increasing and subquadratic, and let $\ $%
\begin{equation}
x_{r}\leq 2\tsum_{i=1}^{n}\lambda _{i}x_{i},\qquad r=1,...,n  \label{3.14}
\end{equation}%
where $\lambda _{i}\geq 0,$\ $x_{i}\geq 0,$\ $i=1,...,n,$\ $%
\tsum_{i=1}^{n}\lambda _{i}=1.$ Then\ \ \ 
\begin{equation}
\tsum_{r=1}^{n}\lambda _{r}f\left( x_{r}\right) \leq 2f\left(
\tsum_{r=1}^{n}\lambda _{r}\left( x_{r}\right) \right).  \label{3.15}
\end{equation}%
In particular, if $\max \{x_{r}:r=1,...,n\}\leq 2\min \{x_{r}:r=1,...,n\}$,
then (\ref{3.14}) holds and therefore (\ref{3.15}) holds.
\end{lemma}

For subquadratic \ increasing functions (which are therefore also
non-negative) the proofs of the following Theorem \ref{T10} for $A_{n}\left(
f\right) $ and $B_{n}\left( f\right) $\ show that (\ref{3.15}) holds as (\ref%
{3.14}) always holds. The bounds given here for subquadratic functions are
not the best possible, but as they are easy to obtain we state these bounds
and prove them.

\begin{theorem}
\label{T10}Let $\ f$\ be increasing subquadratic function on $[0,1].$ \ Then
for $\ n\geq 2$%
\begin{equation}
A_{n+1}\left( f\right) \leq 2A_{n}\left( f\right)  \label{3.16}
\end{equation}%
and 
\begin{equation}
B_{n-1}\left( f\right) \leq 2B_{n}\left( f\right) .  \label{3.17}
\end{equation}%
If$\ f$ is also convex, we get that 
\begin{equation}
A_{n}\left( f\right) \leq A_{n+1}\left( f\right) \leq 2A_{n}\left( f\right)
\label{3.18}
\end{equation}%
and%
\begin{equation}
B_{n}\left( f\right) \leq B_{n-1}\left( f\right) \leq 2B_{n}\left( f\right)
\label{3.19}
\end{equation}
\end{theorem}

\begin{proof}
We use the same technique as in \cite{AJS2} and \cite{BJ}. By some
 manipulations we get that 
\begin{equation*}
A_{n+1}\left( f\right) =\tfrac{1}{n}\tsum_{r=1}^{n}f\left( \tfrac{r}{n+1}%
\right) =\tfrac{1}{n-1}\tsum_{r=1}^{n-1}\left[ \tfrac{r}{n}f\left( \tfrac{r+1%
}{n+1}\right) +\tfrac{n-r}{n}f\left( \tfrac{r}{n+1}\right) \right] .
\end{equation*}%
Let us denote \ \ $x_{1}\left( r\right) =\tfrac{r+1}{n+1},$\ \ $x_{2}\left(
r\right) =\tfrac{r}{n+1},$\ \ $\lambda _{1}=\tfrac{r}{n},$\ \ $\lambda _{2}=%
\tfrac{n-r\ \ }{n},$ \ for \ $1\leq r\leq n-1$. Now we get that $\lambda
_{1}x_{1}+\lambda _{2}x_{2}=\tfrac{r}{n}.$

As $x_{1}=\tfrac{r+1}{n+1}<2\left( \tfrac{r}{n}\right) =2\left( \lambda
_{1}x_{1}+\lambda _{2}x_{2}\right) \ $\ and $x_{2}=\tfrac{r}{n+1}<$\ $%
2\left( \tfrac{r}{n}\right) =2\left( \lambda _{1}x_{1}+\lambda
_{2}x_{2}\right) ,$ \ we get that (\ref{3.14}) is satisfied, and as $\ f$ \
is subquadratic increasing we get by Lemma \ref{L3} that 
\begin{equation*}
\tfrac{1}{n-1}\left[ \tfrac{n-1}{n}\tsum_{r=1}^{n}f\left( \tfrac{r}{n+1}%
\right) \right] \leq 2\tfrac{1}{n-1}\left[ \tsum_{r=1}^{n-1}f\left( \tfrac{r%
}{n}\right) \right] ,
\end{equation*}%
and this is inequality (\ref{3.16}). The same reasoning leads to (\ref{3.17}%
). In \cite{BJ} it was proved that $A_{n}\left( f\right) $ increases with $n$
and $B_{n}(f)$ decreases when $f$ is convex. Therefore, if $f$ is convex
increasing and subquadratic we get that (\ref{3.18}) and (\ref{3.19}) hold.
\end{proof}

\medskip

\begin{remark}
\label{R4}The same reasoning that lead to (\ref{3.16}) \ also shows that $%
A_{n}(f)\leq 2f\left( \tfrac{1}{2}\right) $\ for subquadratic increasing
function, and if $f$ is also convex, \ $A_{n}(f)\geq f\left( \tfrac{1}{2}%
\right) $. Therefore, if $f$ is convex increasing subquadratic function then
$f\left( \tfrac{1}{2}\right) \leq \ A_{n}(f)\leq \ A_{n+1}(f)\leq \
2A_{n}(f)\leq 4f\left( \tfrac{1}{2}\right)$.
\end{remark}

\begin{remark}
\label{R5}The following functions are examples of subquadratic increasing
functions which are also convex (see \cite{AJS1}) and therefore 
satisfy Theorem \ref{T10}: 
\begin{eqnarray*}
f(x) &=&x^{p},\qquad 1\leq p\leq 2, \\
f(x) &=&(1+x^{p})^{\tfrac{1}{p}},\qquad 1\leq p, \\
f(x) &=&(1+x^{p})^{\tfrac{1}{p}}-1,\qquad 1\leq p\leq 2, \\
f(x) &=&3x^{2}-2x^{2}\log (x),\qquad 0\leq x\leq 1.
\end{eqnarray*}
\end{remark}

\section{Superquadraticity and lower bounds of averages}

The following theorem refines the results of \cite{QG} for convex functions
which are also superquadratic (like $f\left( x\right) =x^{m}$, \ $m\geq 2).$

\begin{theorem}
\label{T1}Let $f$ be a positive superquadratic function on $\left[ 0,1\right]
$ and let $(a_{i})_{i\in \mathbb{N}}$, $a_i>0$, and 
$\left(i\left( 1-\tfrac{a_{i}}{a_{i+1}}\right)\right)_{i\in \mathbb{N}}$ 
be increasing sequences. Then for $n\geq 2$ we get that%
\begin{align}
\Delta & :=\tfrac{1}{n}\tsum_{i=1}^{n}f\left( \tfrac{a_{i}}{a_{n}}\right) -%
\tfrac{1}{n+1}\tsum_{i=1}^{n+1}f\left( \tfrac{a_{i}}{a_{n+1}}\right)  \notag
\\
& \geq \tfrac{n-1}{n+1}\left[ f\left( \tfrac{a_{2}-a_{1}}{na_{n}}\right)
+f\left( \tfrac{\left( n-2\right) \left( a_{2}-a_{1}\right) }{2\left(
n-1\right) na_{n}}\right) \text{+\ }f\left( \tfrac{\left( n-2\right) \left(
a_{2}-a_{1}\right) }{3n^{2}a_{n}}\right) \right] .  \label{2.1}
\end{align}
\end{theorem}

\begin{proof}
From the conditions on $a_{i},$ $i=1,...,n+1$, it is obvious that%
\begin{equation}
\tfrac{i}{n}\left( \tfrac{a_{i+1}}{a_{n}}-\tfrac{a_{i}}{a_{n}}\right) \geq 
\tfrac{a_{2}-a_{1}}{na_{n}}=:A.  \label{2.2}
\end{equation}%
By the same considerations as in \cite{QG} we get that%
\begin{equation}
\tfrac{\left( i-1\right) a_{i-1}+\left( n-i+1\right) a_{i}}{na_{n}}\geq 
\tfrac{a_{i}}{a_{n+1}}  \label{2.3}
\end{equation}%
and%
\begin{eqnarray}
\Delta &=&\tfrac{1}{n}\tsum_{i=1}^{n}f\left( \tfrac{a_{i}}{a_{n}}\right) -%
\tfrac{1}{n+1}\tsum_{i=1}^{n+1}f\left( \tfrac{a_{i}}{a_{n+1}}\right)  \notag
\\
&=&\tfrac{1}{n+1}\tsum_{i=1}^{n}\left[ \tfrac{\left( i-1\right) }{n}f\left( 
\tfrac{a_{i-1}}{a_{n}}\right) +\tfrac{\left( n-i+1\right) }{n}f\left( \tfrac{%
a_{i}}{a_{n}}\right) -f\left( \tfrac{a_{i}}{a_{n+1}}\right) \right] .
\label{2.4}
\end{eqnarray}%
From the superquadraticity of $f$ we get from (\ref{1.2}) and (\ref{2.4})
that%
\begin{eqnarray}
\Delta &=&\tfrac{1}{n+1}\tsum_{i=1}^{n}\left[ \tfrac{\left( i-1\right) }{n}%
f\left( \tfrac{a_{i-1}}{a_{n}}\right) +\tfrac{\left( n-i+1\right) }{n}%
f\left( \tfrac{a_{i}}{a_{n}}\right) -f\left( \tfrac{a_{i}}{a_{n+1}}\right) %
\right]  \notag \\
&\geq &\tfrac{1}{n+1}\tsum_{i=1}^{n}\left[ \tfrac{\left( i-1\right) }{n}%
f\left( \tfrac{\left( n-i+1\right) \left( a_{i}-a_{i-1}\right) }{na_{n}}%
\right) +\tfrac{\left( n-i+1\right) }{n}f\left( \tfrac{\left( i-1\right)
\left( a_{i}-a_{i-1}\right) }{na_{n}}\right) \right]  \notag \\
&&+\tfrac{1}{n+1}\tsum_{i=1}^{n}\left[ f\left( \tfrac{\left( i-1\right)
a_{i-1}+\left( n-i+1\right) a_{i}}{na_{n}}\right) -f\left( \tfrac{a_{i}}{%
a_{n+1}}\right) \right]  \label{2.5}
\end{eqnarray}%
where \ we let $a_{0}=0.$

By Lemma A as $f(x)$ is positive superquadratic it is also increasing.
Therefore from (\ref{2.3}) we get that%
\begin{equation}
\tsum_{i=1}^{n}\left[ f\left( \tfrac{\left( i-1\right) a_{i-1}+\left(
n-i+1\right) a_{i}}{na_{n}}\right) -f\left( \tfrac{a_{i}}{a_{n+1}}\right) %
\right] \geq 0,  \label{2.6}
\end{equation}%
and from (\ref{2.5}) and (\ref{2.6}) we get%
\begin{equation}
\Delta \geq \tfrac{1}{n+1}\tsum_{i=1}^{n}\left[ \tfrac{\left( i-1\right) }{n}%
f\left( \tfrac{\left( n-i+1\right) \left( a_{i}-a_{i-1}\right) }{na_{n}}%
\right) +\tfrac{\left( n-i+1\right) }{n}f\left( \tfrac{\left( i-1\right)
\left( a_{i}-a_{i-1}\right) }{na_{n}}\right) \right]  \label{2.7}
\end{equation}%
As $f(0)=0$ it follows%
\begin{eqnarray}
&&\tsum_{i=1}^{n}\left[ \tfrac{\left( i-1\right) }{n}f\left( \tfrac{\left(
n-i+1\right) \left( a_{i}-a_{i-1}\right) }{na_{n}}\right) +\tfrac{\left(
n-i+1\right) }{n}f\left( \tfrac{\left( i-1\right) \left(
a_{i}-a_{i-1}\right) }{na_{n}}\right) \right]  \notag \\
&=&\tsum_{i=1}^{n-1}\left[ \tfrac{i}{n}f\left( \left( 1-\tfrac{i}{n}\right)
\left( \tfrac{a_{i+1}-a_{i}}{a_{n}}\right) \right) +\left( 1-\tfrac{i}{n}%
\right) f\left( \tfrac{i}{n}\tfrac{\left( a_{i+1}-a_{i}\right) }{a_{n}}%
\right) \right]  \label{2.8}
\end{eqnarray}

From (\ref{2.7}),(\ref{2.8}) and Lemma \ref{L2} we have%
\begin{eqnarray}
\Delta &\geq &\tfrac{1}{n+1}\tsum_{i=1}^{n-1}\left[ \tfrac{i}{n}f\left(
\left( 1-\tfrac{i}{n}\right) \left( \tfrac{a_{i+1}-a_{i}}{a_{n}}\right)
\right) +\left( 1-\tfrac{i}{n}\right) f\left( \tfrac{i}{n}\tfrac{\left(
a_{i+1}-a_{i}\right) }{a_{n}}\right) \right]  \notag \\
&\geq &\tfrac{n-1}{n+1}\tsum_{k=0}^{2}f\left( \tfrac{\tsum_{i=1}^{n-1}2\tfrac{i%
}{n}\left( 1-\tfrac{i}{n}\right) \left\vert 1-\tfrac{2i}{n}\right\vert
^{k}\left( \tfrac{a_{i+1}-a_{i}}{a_{n}}\right) }{n-1}\right)  \label{2.9}
\end{eqnarray}%
holds by inserting in (\ref{1.6}) $\lambda _{i}=\tfrac{i}{n},$ $A_{i}=\tfrac{%
a_{i+1}-a_{i}}{a_{n}},$ $t=2.$ We also get from (\ref{1.7}) that%
\begin{equation}
\Delta \geq \tfrac{n-1}{n+1}\tsum_{k=0}^{2}f\left( \tfrac{%
\tsum_{i=1}^{n-1}2\left( n-i\right) \left\vert n-2i\right\vert ^{k}\left(
a_{2}-a_{1}\right) }{n^{2+k}\left( n-1\right) a_{n}}\right) .  \label{2.10}
\end{equation}%
It is easy to verify that $\tsum_{i=1}^{n-1}2\left( n-i\right) =n\left( n-1\right)$ and
\begin{equation*}
\tsum_{i=1}^{n-1}2\left( n-i\right) \left\vert n-2i\right\vert =\left\{ 
\begin{array}{cc}
\tfrac{\left( n-2\right) n^{2}}{2}, & \left[ \tfrac{n}{2}\right] =\tfrac{n}{2%
} \\ 
\tfrac{n\left( n-1\right) ^{2}}{2}, & \left[ \tfrac{n}{2}\right] =\tfrac{n-1%
}{2},%
\end{array}%
\right.
\end{equation*}%
as well as $\tsum_{i=1}^{n-1}2\left( n-i\right) \left( n-2i\right) ^{2}=\tfrac{%
n^{2}\left( n-1\right) \left( n-2\right) }{3}$ and 
$\tfrac{\left( n-2\right) n^{2}}{2}<\tfrac{n\left( n-1\right) ^{2}}{2}, n=2,3,...$.

Hence, 
\begin{eqnarray*}
\Delta &\geq &\tfrac{1}{n}\tsum_{i=1}^{n}f\left( \tfrac{a_{i}}{a_{n}}\right)
-\tfrac{1}{n+1}\tsum_{i=1}^{n+1}f\left( \tfrac{a_{i}}{a_{n+1}}\right) \\
&\geq &\tfrac{n-1}{n+1}\left[ f\left( \tfrac{a_{2}-a_{1}}{na_{n}}\right)
+f\left( \tfrac{\left( n-2\right) \left( a_{2}-a_{1}\right) }{2\left(
n-1\right) na_{n}}\right) +f\left( \tfrac{\left( n-2\right) \left(
a_{2}-a_{1}\right) }{3n^{2}a_{n}}\right) \right] .
\end{eqnarray*}%
This completes the proof of the theorem.
\end{proof}

The following theorem deals with a lower bound for 
\begin{equation*}
\tfrac{1}{c_{n}}\tsum_{i=1}^{n}f\left( \tfrac{a_{i}}{a_{n}}\right) -\tfrac{1%
}{c_{n+1}}\tsum_{i=1}^{n+1}f\left( \tfrac{a_{i}}{a_{n+1}}\right)
\end{equation*}%
under the same conditions as in \cite[Theorem 5.6]{AJS2}, where a
different lower bound is obtained for positive superquadratic function $f$. 

% Also, in 
% \cite[Theorem 5.4(i)]{AJS2} it is proved that \ $\tfrac{1}{c_{n}}%
% \tsum_{i=1}^{n}f\left( \tfrac{a_{i}}{a_{n}}\right) $\ is decreasing when $f$
% is convex non-negative and increasing, under the same conditions on the
% sequences $c_{n}$ and $a_{n}$\ as in the following Theorem \ref{T2} here.

\begin{theorem}
\label{T2}Let $\left( a_{i}\right) _{i\geq 0}$ and $\left( c_{i}\right)
_{i\geq 0}$ be sequences such that \ $a_{i}>0$, $c_{i}>0$, for $i\geq
1$, \ and\newline
(I) $\ (c_{i})_{i\in \mathbb{N}}$ \ is increasing and $c_{0}=0$,\newline
(II) $\ \left( c_{i}-c_{i-1}\right)_{i\in \mathbb{N}}$ \ is increasing,\newline
(III) $\ c_{1}\left( 1-\tfrac{a_{1}}{a_{2}}\right) \leq c_{i-1}\left( 1-\tfrac{%
a_{i-1}}{a_{i}}\right) \leq c_{n}\left( 1-\tfrac{a_{n}}{a_{n+1}}\right)$, $i=1,...,n$, $%
n\geq 1$,\newline
(IV) $\ a_{0}=0$ \ and\ \ $\left( a_{i}\right)_{i\in \mathbb{N}} $ is increasing.\newline
% Given a function $f$, let 
% \begin{equation*}
% S_{n}\left( f,\left( a_{n}\right) ,\left( c_{n}\right) \right) =S_{n}\left(
% f\right) :=\tfrac{1}{c_{n}}\tsum_{r=1}^{n}f\left( \tfrac{a_{r}}{a_{n}}%
% \right) ,\text{ \ \ }n\geq 1.
% \end{equation*}%
\ If f is superquadratic and non-negative function on $\left[ 0,1\right] ,$ then%
\begin{equation*}
D:=\tfrac{1}{c_{n}}\tsum_{r=1}^{n}f\left( \tfrac{a_{r}}{a_{n}}
 \right) - \tfrac{1}{c_{n+1}}\tsum_{r=1}^{n+1}f\left( \tfrac{a_{r}}{a_{n+1}}
 \right)  \geq \tfrac{n-1}{c_{n+1}}%
f\left( \tfrac{2c_{1}\left( a_{2}-a_{1}\right) \left( c_{n}-c_{n-1}\right) }{%
c_{n}^{2}a_{n}}\right), \quad n\geq 1 .
\end{equation*}
\end{theorem}

\begin{proof}
From the conditions (I)-(IV) it is clear that 
\begin{equation}
\tfrac{c_{i}}{c_{n}}\left( \tfrac{a_{i+1}}{a_{n}}-\tfrac{a_{i}}{a_{n}}%
\right) \geq \tfrac{c_{1}\left( a_{2}-a_{1}\right) }{a_{n}c_{n}}=:A
\label{2.11}
\end{equation}%
and that%
\begin{equation}
\tfrac{c_{i-1}a_{i-1}+a_{i}\left( c_{n}-c_{i-1}\right) }{a_{n}c_{n}}\geq 
\tfrac{a_{i}}{a_{n+1}}, \quad \text{for an arbitrary} \quad i=1,...,n.  \label{2.12}
\end{equation}%
From Lemma A we know that $f(0)\leq 0$, therefore if $f$ is positive then $f(0)=0$
and from (I) and (II) we get 
\begin{eqnarray}
D &=&\tfrac{1}{c_{n+1}}\left\{ \tsum_{i=1}^{n}\left[ \tfrac{c_{i}}{c_{n}}%
f\left( \tfrac{a_{i}}{a_{n}}\right) +\tfrac{\left( c_{n+1}-c_{i}\right) }{%
c_{n}}f\left( \tfrac{a_{i}}{a_{n}}\right) \right] -\tsum_{i=1}^{n+1}f\left( 
\tfrac{a_{i}}{a_{n+1}}\right) \right\}  \notag \\
&=&\tfrac{1}{c_{n+1}}\left\{ \tsum_{i=1}^{n}\left[ \tfrac{c_{i-1}}{c_{n}}%
f\left( \tfrac{a_{i-1}}{a_{n}}\right) +\tfrac{\left( c_{n+1}-c_{i}\right) }{%
c_{n}}f\left( \tfrac{a_{i}}{a_{n}}\right) \right] -\tsum_{i=1}^{n}f\left( 
\tfrac{a_{i}}{a_{n+1}}\right) \right\}  \notag \\
&\geq &\tfrac{1}{c_{n+1}}\left\{ \tsum_{i=1}^{n}\left[ \tfrac{c_{i-1}}{c_{n}}%
f\left( \tfrac{a_{i-1}}{a_{n}}\right) +\tfrac{\left( c_{n}-c_{i-1}\right) }{%
c_{n}}f\left( \tfrac{a_{i}}{a_{n}}\right) \right] -\tsum_{i=1}^{n}f\left( 
\tfrac{a_{i}}{a_{n+1}}\right) \right\}  \label{2.13}
\end{eqnarray}%
Using the superquadraticity of $f$ from (\ref{1.2}) and (\ref{2.13}) it follows%
\begin{eqnarray}
D &\geq &\tfrac{1}{c_{n+1}}\tsum_{i=1}^{n}\left[ \tfrac{c_{i-1}}{c_{n}}%
f\left( \tfrac{\left( c_{n}-c_{i-1}\right) \left( a_{i}-a_{i-1}\right) }{%
a_{n}c_{n}}\right) +\tfrac{\left( c_{n}-c_{i-1}\right) }{c_{n}}f\left( 
\tfrac{c_{i-1}\left( a_{i-}a_{i-1}\right) }{c_{n}a_{n}}\right) \right] 
\notag \\
&&+\tfrac{1}{c_{n+1}}\left[ \tsum_{i=1}^{n}f\left( \tfrac{%
c_{i-1}a_{i-1}+\left( c_{n}-c_{i-1}\right) a_{i}}{c_{n}a_{n}}\right)
-\tsum_{i=1}^{n}f\left( \tfrac{a_{i}}{a_{n+1}}\right) \right] .  \label{2.14}
\end{eqnarray}%
As $f$ is increasing, according to (\ref{2.12}), (\ref{2.13}) implies
\begin{eqnarray}
D &\geq &\tfrac{1}{c_{n+1}}\tsum_{i=1}^{n}\left[ \tfrac{c_{i-1}}{c_{n}}%
f\left( \tfrac{\left( c_{n}-c_{i-1}\right) \left( a_{i}-a_{i-1}\right) }{%
a_{n}c_{n}}\right) +\tfrac{\left( c_{n}-c_{i-1}\right) }{c_{n}}f\left( 
\tfrac{c_{i-1}\left( a_{i}-a_{i-1}\right) }{a_{n}c_{n}}\right) \right] 
\notag \\
&=&\tfrac{1}{c_{n+1}}\tsum_{i=1}^{n-1}\left[ \tfrac{c_{i}}{c_{n}}f\left( 
\tfrac{\left( c_{n}-c_{i}\right) \left( a_{i+1}-a_{i}\right) }{a_{n}c_{n}}%
\right) +\tfrac{\left( c_{n}-c_{i}\right) }{c_{n}}f\left( \tfrac{c_{i}\left(
a_{i+1-}a_{i}\right) }{c_{n}a_{n}}\right) \right] .  \label{2.15}
\end{eqnarray}%
The last equality follows from $f(0)=0$ and $c_{0}=0.$

Inserting in (\ref{1.6}) that $\lambda _{i}=\tfrac{c_{i}}{c_{n}}$ \ and \ $%
A_{i}=\tfrac{a_{i+1}-a_{i}}{a_{n}}$, from (\ref{2.15}), we get 
\begin{equation}
D\geq \tfrac{n-1}{c_{n+1}}\tsum_{k=0}^{t}f\left( \tfrac{\tsum_{i=1}^{n-1}2%
\left( \tfrac{c_{i}}{c_{n}}\right) \left( 1-\tfrac{c_{i}}{c_{n}}\right)
\left\vert 1-\tfrac{2c_{i}}{c_{n}}\right\vert ^{k}\left( \tfrac{a_{i+1}-a_{i}}{%
a_{n}}\right) }{n-1}\right).  \label{2.16}
\end{equation}%

Using (\ref{2.11}) and (\ref{1.7}) we get from (\ref{2.16}) that%
\begin{equation}
D\geq \tfrac{n-1}{c_{n+1}}\tsum_{k=0}^{t}f\left( \tfrac{\tsum_{i=1}^{n-1}2%
\left( 1-\tfrac{c_{i}}{c_{n}}\right) \left\vert 1-\tfrac{2c_{i}}{c_{n}}%
\right\vert ^{k}c_{1}\left( a_{2}-a_{1}\right) }{\left( n-1\right) a_{n}c_{n}%
}\right) ,\text{ \ \ }t=0,1,2,... . \label{2.17}
\end{equation}%
As $(c_{i})_{i\in \mathbb{N}}$ is increasing, from (\ref{2.17}) we have
\begin{align*}
D& =\tfrac{1}{c_{n}}%
\tsum_{i=1}^{n}f\left( \tfrac{a_{i}}{a_{n}}\right) -\tfrac{1}{c_{n+1}}%
\tsum_{i=1}^{n+1}f\left( \tfrac{a_{i}}{a_{n+1}}\right) \\
& \geq \tfrac{n-1}{c_{n+1}}\tsum_{k=0}^{t}f\left( \tfrac{\tsum_{i=1}^{n-1}2%
\left( c_{n}-c_{n-1}\right) \left\vert c_{n}-2c_{i}\right\vert
^{k}c_{1}\left( a_{2}-a_{1}\right) }{\left( n-1\right) c_{n}^{2+k}a_{n}}%
\right) \\
& \geq \tfrac{n-1}{c_{n+1}}f\left( \tfrac{2c_{1}\left( a_{2}-a_{1}\right)
\left( c_{n}-c_{n-1}\right) }{c_{n}^{2}a_{n}}\right) .
\end{align*}
\end{proof}

\begin{remark}
\label{R2} If $\left(c_{i}\right)_{i\in \mathbb{N}}=(a_{i})_{i\in \mathbb{N}}$, from Theorem \ref{T2}, we get a
refinement of Theorem 2 in \cite{QG} \ for convex functions that are also
superquadratic.
\end{remark}

In the following theorem we extend our investigation to three
sequences. Investigation with three sequences was also dealt in \cite{AJS2}.

\begin{theorem}
\label{T3}Let $f$ be a positive superquadratic function on $\left[ 0,L\right]
.$\ \ Let $\left( a_{i}\right) _{i\geq 0},$ $\left( b_{i}\right) _{i\geq 0},$
$\left( c_{i}\right) _{i\geq 0}$\ \ be sequences such that \ $a_{i}>0,$ $%
b_{i}>0,$ $c_{i}>0$\ for $i \geq 1$\ and\newline
(a) $\ \left(a_{i}\right)_{i\in \mathbb{N}},$ $\left(b_{i}\right)_{i\in \mathbb{N}},$ 
$\left(c_{i}\right)_{i\in \mathbb{N}}$\ are increasing and $a_{0}=c_{0}=0,$\newline
(b) $\left(c_{i}-c_{i-1}\right)_{i\in \mathbb{N}}$\ is increasing,\newline
(c) $\ c_{n}\left( 1-\tfrac{b_{n}}{b_{n+1}}\right) \geq c_{r}\left( 1-\tfrac{%
a_{r}}{a_{r+1}}\right) ,$ for $r\leq n.$\newline
Then%
\begin{eqnarray}
H&:=&\tfrac{1}{c_{n}}%
\tsum_{r=1}^{n}f\left( \tfrac{a_{r}}{b_{n}}\right) -\tfrac{1}{c_{n+1}}%
\tsum_{r=1}^{n+1}f\left( \tfrac{a_{r}}{b_{n+1}}\right)  \notag \\
&\geq & \tfrac{n-1}{c_{n+1}}f\left( \tfrac{2c_{1}\left( c_{n}-c_{n-1}\right) A%
}{c_{n}^{2}b_{n}}\right) ,  \label{2.18}
\end{eqnarray}%
where $A:=\min \{a_{i+1}-a_{i}:i=1,...,n\}$.
\end{theorem}

\begin{proof}
The technique of the proof is analogue to the techniques that we used in proving Theorem \ref{T1}
and Theorem \ref{T2}.
\end{proof}

\section{Subquadraticity and averages}

In this chapter we deal with functions that are increasing and subquadratic
on $\left[ 0,1\right] $, like are $f(x)=x^{m}$, $0\leq m\leq 2$ and 
\begin{equation*}
f\left( x\right) =\left\{ 
\begin{array}{ll}
x^{2}-2x^{2}\log x, & 0<x\leq 1 \\ 
0\qquad \qquad \qquad , & x=0%
\end{array}%
\right. .
\end{equation*}%
The last function is not concave and not convex and therefore none of the
results of \cite{BJ}, \cite{CQCD}, \cite{EP}, and \cite{QG} are applicable
to this function.\medskip

% In the following theorems we use Lemma 1 inequality (\ref{1.5}).

\begin{theorem}
\label{T8}Let f be increasing subquadratic function on $\left[ 0,1\right] .$%
\ \ Let $\left( a_{i}\right) _{i\geq 0}$\ \ satisfy\linebreak (A) 
$\left(a_i\right)_{i\in \mathbb{N}}$ is increasing sequence and $a_{i}>0$, $i=1,...,n+1$,\newline
(B) $\ i\left( \tfrac{a_{i+1}}{a_{i}}-1\right) \leq n\left( \tfrac{a_{n+1}}{%
a_n}-1\right) ,$ \ $i=1,...,n.$\newline
Then%
\begin{eqnarray}
E&:=&\tfrac{1}{n+1}
\tsum_{i=1}^{n+1}f\left( \tfrac{a_{i}}{a_{n+1}}\right) -\tfrac{1}{n}%
\tsum_{i=1}^{n}f\left( \tfrac{a_{i}}{a_{n}}\right)  \notag \\
&\leq &\tsum_{i=1}^{n}\left[ \tfrac{i}{n\left( n+1\right) }f\left( \tfrac{%
n-i+1}{n+1}\cdot \tfrac{a_{i+1}-a_{i}}{a_{n+1}}\right) +\tfrac{n-i+1}{%
n\left( n+1\right) }f\left( \tfrac{i}{n+1}\cdot \tfrac{a_{i+1}-a_{i}}{a_{n+1}%
}\right) \right] .  \label{3.1}
\end{eqnarray}%
Moreover, if in addition\newline
(C) $\ \tfrac{a_{i+1}}{a_{i}}\leq 2,$ \ $i=1,...,n$\newline
holds, then%
\begin{equation}
\tfrac{1}{n+1}\tsum_{i=1}^{n+1}f\left( \tfrac{a_{i}}{a_{n+1}}\right) \leq 
\tfrac{2}{n}\tsum_{i=1}^{n}f\left( \tfrac{a_{i}}{a_{n}}\right).  \label{3.2}
\end{equation}
\end{theorem}

\begin{proof}
Since $\left( a_{i}\right) _{i\in \mathbb{N}}$ increases we get from (B)\ that
\begin{equation*}
\left( n+1\right) \left( \tfrac{a_{n+1}}{a_{n}}-1\right) \geq n\left( \tfrac{%
a_{n+1}}{a_{n}}-1\right) \geq i\left( \tfrac{a_{i+1}}{a_{i}}-1\right) ,\text{
\ \ \ \ \ }i=1,...,n,
\end{equation*}
which is equivalent to 
\begin{equation}
\tfrac{ia_{i+1}+\left( n-i+1\right) a_{i}}{\left( n+1\right) a_{n+1}}\leq 
\tfrac{a_{i}}{a_{n}}  \label{3.3}
\end{equation}%
Rewriting E we get%
\begin{eqnarray}
E &=&\tfrac{1}{n}\left[ \tsum_{i=1}^{n+1}\tfrac{n}{n+1}f\left( \tfrac{a_{i}}{%
a_{n+1}}\right) -\tsum_{i=1}^{n}f\left( \tfrac{a_{i}}{a_{n}}\right) \right] 
\notag \\
&=&\tfrac{1}{n}\left[ \tsum_{i=1}^{n+1}\tfrac{i-1}{n+1}f\left( \tfrac{a_{i}}{%
a_{n+1}}\right) +\tsum_{i=1}^{n+1}\tfrac{n-i+1}{n+1}f\left( \tfrac{a_{i}}{%
a_{n+1}}\right) -\tsum_{i=1}^{n}f\left( \tfrac{a_{i}}{a_{n}}\right) \right] 
\notag \\
&=&\tfrac{1}{n}\tsum_{i=1}^{n}\left[ \tfrac{i}{n+1}f\left( \tfrac{a_{i+1}}{%
a_{n+1}}\right) +\tfrac{n-i+1}{n+1}f\left( \tfrac{a_{i}}{a_{n+1}}\right)
-f\left( \tfrac{a_{i}}{a_{n}}\right) \right] .  \label{3.4}
\end{eqnarray}%
As $f$ is subquadratic, by using (\ref{1.6}) we get from (\ref{3.4}) that%
\begin{eqnarray}
E &=&\tfrac{1}{n}\left\{ \tsum_{i=1}^{n}\left[ \tfrac{i}{n+1}f\left( \tfrac{%
a_{i+1}}{a_{n+1}}\right) +\tfrac{n-i+1}{n+1}f\left( \tfrac{a_{i}}{a_{n+1}}%
\right) \right] -\tsum_{i=1}^{n}f\left( \tfrac{a_{i}}{a_{n}}\right) \right\}
\notag \\
&\leq &\tfrac{1}{n}\tsum_{i=1}^{n}\left[ \tfrac{i}{n+1}f\left( \tfrac{n-i+1}{%
n+1}\cdot \tfrac{a_{i+1}-a_{i}}{a_{n+1}}\right) +\tfrac{n-i+1}{n+1}f\left( 
\tfrac{i}{n+1}\cdot \tfrac{a_{i+1}-a_{i}}{a_{n+1}}\right) \right]  \notag \\
&&+\tfrac{1}{n}\tsum_{i=1}^{n}\left[ f\left( \tfrac{ia_{i+1}+\left(
n-i+1\right) a_{i}}{\left( n+1\right) a_{n+1}}\right) -f\left( \tfrac{a_{i}}{%
a_{n}}\right) \right] .  \label{3.5}
\end{eqnarray}%
As $f$ is increasing, using (\ref{3.3}), we get from (\ref{3.5}) that%
\begin{equation*}
E\leq \tsum_{i=1}^{n}\left[ \tfrac{i}{n\left( n+1\right) }f\left( \tfrac{%
n-i+1}{n+1}\cdot \tfrac{a_{i+1}-a_{i}}{a_{n+1}}\right) +\tfrac{n-i+1}{%
n\left( n+1\right) }f\left( \tfrac{i}{n+1}\cdot \tfrac{a_{i+1}-a_{i}}{a_{n+1}%
}\right) \right] .
\end{equation*}%
Hence (\ref{3.1}) is proved. If (C) is also satisfied, then it is easy to
verify that%
\begin{equation}
\tfrac{n-i+1}{n+1}\cdot \tfrac{a_{i+1}-a_{i}}{a_{n+1}}\leq \tfrac{%
ia_{i+1}+\left( n-i+1\right) a_{i}}{\left( n+1\right) a_{n+1}}.  \label{3.6}
\end{equation}%
As it is given that $a_{i}>0,$ $i=1,...,n+1$, it is obvious that also 
\begin{equation}
\tfrac{i\left( a_{i+1}-a_{i}\right) }{\left( n+1\right) a_{n+1}}\leq \tfrac{%
ia_{i+1}+\left( n-i+1\right) a_{i}}{\left( n+1\right) a_{n+1}},\text{ \ \ }%
i=1,...,n+1.  \label{3.7}
\end{equation}%
By (\ref{3.6}) and (\ref{3.7}) \ we get from (\ref{3.1}) and (\ref{3.3})
that $E\leq \tfrac{1}{n}\tsum_{i=1}^{n}f\left( \tfrac{a_{i}}{a_{n}}\right)$
\ which is the same as (\ref{3.2}). Hence the theorem is proved.
\end{proof}

\begin{theorem}
\label{T9}Let $f$ be an increasing subquadratic function on $\left[ 0,1%
\right] .$ \ Let $\left( a_{i}\right) _{i\geq 0}$ satisfy\newline
(i) $\ a_{i}>0,$ \ $i=1,...,n+1$ and $a_{0}=0,$\newline
(ii) $\left(a_{i}\right)_{i\in \mathbb{N}}$ and $\left(a_{i}-a_{i-1}\right)_{i\in \mathbb{N}}$
are increasing sequences.\newline
Then%
\begin{eqnarray}
R&:=&\tfrac{1}{a_{n+1}}%
\tsum_{i=1}^{n+1}f\left( \tfrac{a_{i}}{a_{n+1}}\right) -\tfrac{1}{a_{n}}%
\tsum_{i=1}^{n}f\left( \tfrac{a_{i}}{a_{n}}\right)  \notag \\
&\leq &\tfrac{1}{a_{n}}\tsum_{i=1}^{n}\left[ \tfrac{a_{n+1}-a_{i}}{a_{n+1}}%
f\left( \tfrac{a_{i}}{a_{n}}\cdot \tfrac{a_{i+1}-a_{i}}{a_{n+1}}\right) +%
\tfrac{a_{i}}{a_{n+1}}f\left( \tfrac{a_{n}-a_{i}}{a_{n+1}}\cdot \tfrac{%
a_{i+1}-a_{i}}{a_{n+1}}\right) \right] .  \label{3.8}
\end{eqnarray}%
Moreover, if an addition\newline
(iii) $\ \tfrac{a_{i+1}}{a_{i}}\leq 2,$ \ $i=1,...,n$\newline
then%
\begin{equation}
\tfrac{1}{a_{n+1}}\tsum_{i=1}^{n+1}f\left( \tfrac{a_{i}}{a_{n+1}}\right)
\leq \tfrac{2}{a_{n}}\tsum_{i=1}^{n}f\left( \tfrac{a_{i}}{a_{n}}\right) .
\label{3.9}
\end{equation}
\end{theorem}

\begin{proof}
The steps of the proof are analogue to the steps we made in proving Theorem \ref{T8}.
\end{proof}

\end{document}